\documentclass[a4paper,twoside,11pt]{amsproc} %

\usepackage[margin=2cm,a4paper]{geometry} 
\usepackage{pgf,tikz}
\usetikzlibrary{arrows}

\usepackage{amsmath, amssymb, amsthm, bm, centernot}
\usepackage{color,booktabs}
\usepackage{url, hyperref}
\usepackage{enumerate}

\newtheorem{thm}{Theorem}

\newtheorem{lem}[thm]{Lemma}
\newtheorem{cor}[thm]{Corollary}

\numberwithin{thm}{section}

\theoremstyle{definition}
\newtheorem*{mydef}{Definition}

\theoremstyle{remark}
\newtheorem{rem}{Remark}

\theoremstyle{definition}

\newcommand{\m}[1]{\mathcal{#1}}

\newcommand{\dist}{\mathrm{d}}
\newcommand\GH{\mathsf{H}}
\newcommand\Q{\mathsf{Q}}
\newcommand\PG{\mathsf{PG}}
\newcommand\sympl{\mathsf{W}}
\DeclareMathOperator\inc{\mathrm{I}}
\DeclareMathOperator\notinc{\centernot{\mathrm{I}}}

\renewcommand\le{\leqslant}
\renewcommand\ge{\geqslant}

\mathchardef\mh="2D

\title{On regular induced subgraphs of generalized polygons}
\author{John Bamberg, Anurag Bishnoi, Gordon F. Royle}

\address[J. Bamberg and G. F. Royle]{
Centre for the Mathematics of Symmetry and Computation\\
The University of Western Australia\\
Australia.}
\email{ \{John.Bamberg,Gordon.Royle\}@uwa.edu.au}
\address[A. Bishnoi]{
Department of Mathematics\\
Ghent University\\
Belgium.}
\email{anurag.2357@gmail.com}



\begin{document}
\maketitle
\begin{abstract}

The cage problem asks for the smallest number $c(k,g)$ of vertices in a $k$-regular graph of girth
$g$ and graphs meeting this bound are known as \emph{cages}. While cages are known
to exist for all integers $k \ge 2$ and $g \ge 3$, the exact value of $c(k, g)$ is known only for
some small values of $k, g$ and three infinite families where $g \in \{6, 8, 12\}$ and $k - 1$ is a
prime power. These infinite families come from the incidence graphs of generalized polygons.  Some
of the best known upper bounds on $c(k,g)$ for $g \in \{6, 8, 12\}$ have been obtained by constructing small regular induced subgraphs
of these cages.

In this paper, we first use the Expander Mixing Lemma to give a general lower bound on the size of an
induced $k$-regular subgraph of a regular bipartite graph in terms of the second largest eigenvalue
of the host graph.  We use this bound to show that the known construction of $(k,6)$-graphs using
Baer subplanes of the Desarguesian projective plane is the best possible.  For generalized
quadrangles and hexagons, our bounds are new.  In particular, we improve the known lower bound on
the size of an induced $q$-regular subgraphs of the classical generalized quadrangle $\Q(4,q)$
and show that the known constructions are asymptotically sharp, which answers a question of Metsch
\cite[Section 6]{Metsch11}.

For prime powers $q$, we also improve the known upper bounds on $c(q,8)$ and $c(q,12)$ by giving new
geometric constructions of $q$-regular induced subgraphs in the symplectic generalized quadrangle
$\sympl(3,q)$ and the split Cayley hexagon $\GH(q)$, respectively.  Our constructions show that
\[c(q,8) \le 2(q^3 - q\sqrt{q} - q)\]
for $q$ an even power  of a prime, and 
\[c(q, 12) \le 2(q^5 - 3q^3)\]
for all prime powers $q$.  For $q \in \{3,4,5\}$ we also give a computer classification of all
$q$-regular induced subgraphs of the classical generalized quadrangles of order $q$.
For $\sympl(3,7)$ we classify all $7$-regular induced subgraphs which have a non-trivial automorphism. 
\end{abstract}

\section{Introduction}
\label{sec:intro}

A $(k,g)$-graph is a simple undirected graph which is $k$-regular and has girth $g$ (length of the
shortest cycle).  The study of the cage problem begins with the observation that a $(k,g)$-graph has
at least $1 + k + k(k-1) + \dots + k(k-1)^{(g - 3)/2}$ vertices for $g$ odd and at least $2(1 + (k
-1) + (k-1)^2 + \dots + (k-1)^{(g - 2)/2})$ vertices for $g$ even.  The $(k,g)$-graphs which meet
these bounds are known as \textit{Moore graphs}.  Interestingly, there are very few Moore graphs.
It has been proved using linear algebraic methods that a Moore graph can only exist in the following
cases: (a) $k = 2$ and $g \ge 3$ (cycles), (b) $g = 3$ and $k \ge 2$ (complete graphs), (c) $g = 4$
and $k \ge 2$ (complete bipartite graphs), (d) $g = 5$ and $k \in \{2, 3, 7, 57\}$, (e) $g \in \{6,
8, 12\}$ and there exists a generalized $(g/2)$-gon of order $k - 1$ \cite{Hoffman-Singleton60,
  Feit-Higman, Bannai-Ito73, Damerell73}.  The existence of a Moore graph with $k = 57$ and $g =
5$ is a famous open problem in graph theory.  Generalized $n$-gons are certain point-line geometries
introduced by Tits \cite{Tits59}, and those with order $k - 1$ are known to exist only when $k - 1$
is a prime power and $n \in \{3, 4, 6\}$ (see Section \ref{sec:prelim} for their definition and
properties).

In view of this scarcity of Moore graphs, a natural problem is to find the minimum number of
vertices in a $k$-regular graph of girth $g$ for arbitrary integers $k \ge 3$ and $g \ge 5$.  This
minimum number is denoted by $c(k, g)$ and the graphs with $c(k, g)$ vertices are known as
\emph{cages}. The problem of determining $c(k,g)$ is then called \textit{the cage problem}.  It was
shown by Erd\H{o}s and Sachs that $c(k,g)$ is finite, that is, cages exist for every possible value of
$k$ and $g$ \cite{Erdos-Sachs63}.  Beyond the Moore graphs, $c(k, g)$ is known exactly for only a
few small cases and we refer to the survey \cite[Section 2]{Excoo-Jajkay11} for a description of
these graphs.  The general problem of determining $c(k,g)$ appears to be extremely hard.  Therefore,
much research has been devoted to obtaining good upper bounds on $c(k, g)$ by constructing small
graphs of given girth and regularity (see \cite[Section 4]{Excoo-Jajkay11} for the state of the art).
In this paper, we will be focussing on the case $g \in \{6, 8, 12\}$ when $k - 1$ is not necessarily
a prime power. For general values of $k$ and $g$, the best known upper bounds are due to Lazebnik,
Ustimenko and Woldar \cite{LUW97}.

In \cite{Brown67}, Brown initiated the idea of constructing regular induced subgraphs of known cages
(or Moore graphs) to obtain upper bounds on the number $c(k,g)$.  It was shown that $c(k, 6) < 4k^2$
for all $k$ by taking a prime $p$ satisfying $k < p < 2k$ and then constructing a $k$-regular
subgraph of the incidence graph of the projective plane $\PG(2,p)$ by removing some well chosen
points and lines of the projective plane.  
More constructions for the projective plane were then given in \cite{AFLN06} and some of the bounds were improved.  
In \cite{AGMS07}, further new constructions were given which in particular improved the bounds for $g = 12$.

To give a common treatment of these constructions, G\'{a}cs and H\'{e}ger \cite{Gacs-Heger08}
introduced the notion of a \textit{$t$-good structure} in a generalized $n$-gon, which is a
collection $\m P$ of points and a collection $ \m L$ of lines with the property that every point
outside $\m P$ is incident with exactly $t$ lines of $\m L$ and every line outside $\m L$ is
incident with exactly $t$ points of $\m P$.  Or equivalently, a $t$-good structure in a generalized
polygon $\Gamma$ of order $q$ is the point-line substructure obtained by removing the vertices of an
induced $(q + 1 - t)$-regular subgraph of the incidence graph of $\Gamma$.  Therefore, large
$t$-good structures correspond to small $(q + 1 - t)$-regular subgraphs.  G\'{a}cs and H\'{e}ger
constructed $t$-good structures with $t(q + \sqrt{q} + 1)$ points (and necessarily equally many
lines) in $\PG(2,q)$, for $q$ an even power of a prime, by taking $t$ disjoint Baer subplanes
\cite{Gacs-Heger08}.  They also showed that for all $t \le 2 \sqrt{q}$ the size of any $t$-good
structure in $\PG(2,q)$ is at most $t(q + \sqrt{q} + 1)$, thus proving that their construction is
the best possible for small enough $t$ \cite[Theorem 3.9]{Gacs-Heger08}.  In Section
\ref{sec:bounds}, we will prove that this holds true for all feasible values of $t$.  We also obtain
upper bounds for the sizes of $t$-good structures in generalized quadrangles and hexagons by proving
a general lower bound on the number of vertices in a regular induced subgraph of a regular graph,
which follows from the Expander Mixing Lemma (see Theorem \ref{thm:main_tgood}).  These bounds give
us a limit on the best upper bounds on $c(k,g)$ that can be obtained, for $g \in \{6,8,12\}$, by
this construction method.  Our bound on generalized quadrangles in particular answers a question of
Metsch \cite[Section 6]{Metsch11}.

For constructions of small $(k, g)$-graphs, we focus on $1$-good structures in generalized
quadrangles and hexagons which will allow us to obtain new upper bounds on $c(q, 8)$ and $c(q, 12)$
for prime powers $q$.  We remark that in all of our constructions it is straightforward to prove
using the Moore bound that the $q$-regular induced subgraph that we construct also has the girth of
the original graph. In fact, even if the girth were bigger, we would get an upper bound on $c(q, 8)$
and $c(q,12)$ because of the inequality $c(k, g_1) < c(k,g_2)$ for $g_1 < g_2$ \cite[Theorem
  1]{Fu-Huang-Rodger97}.

For generalized quadrangles, $1$-good structures were studied extensively by Beukemann and Metsch in
\cite{BM11}, where they gave several constructions of $1$-good structures in the classical
generalized quadrangle related to the quadric $\Q(4, q)$ in $\PG(4,q)$ and showed that classifying
such structures appears to be a difficult problem.  Their best constructions have $q^2 + 3q + 1$
points for odd $q$ and $q^2 + 4q + 3$ points for even $q$, which imply the bounds
\[
c(q, 8) \le
\begin{cases}
2(q^3 - 2q),& \text{if } q \text{ is odd}\\
2(q^3 - 3q - 2),& \text{if } q \text{ is even}.
\end{cases}
\]
In Section \ref{sec:const}, we improve the bounds on $c(q, 8)$ obtained by Beukemann and Metsch to 
\[c(q,8) \le 2(q^3 - q\sqrt{q} - q)\]
whenever $q$ is a square, by constructing $1$-good structures in the generalized quadrangle
$\sympl(3,q)$ corresponding to a symplectic form in $\PG(3,q)$ of size $q^2 + q \sqrt{q} + 2q + 1$.

For generalized hexagons of order $q$, the best known $1$-good structures have $q^4 + 2q^3 + q^2 + q
+ 1$ points in them \cite[Construction 2.3]{Gacs-Heger08}.  In Section \ref{sec:const2}, we will
give a general construction that will give us $1$-good structures in $\GH(q)$ of sizes $q^4 + q^3 +
q^2 + q + 1 + k$ for
\[
k\in \{ 0, q^3-q, q^3, q^3+q^2-q, 2q^3-q^2-q, 2q^3-q^2, 2q^3-q, 2q^3, 3q^3-q^2-q, 3q^3-q^2, 3q^3\}.
\]
While two of our examples, of sizes $q^4 + q^3 + q^2 + q + 1$ and $q^4 + 2q^3 + q^2 + q + 1$, are
known, all the other examples are new.  We note that unlike all the constructions of $t$-good
structures in generalized hexagons so far, our construction relies on the geometry of the split
Cayley generalized hexagon $\GH(q)$ represented inside the quadric $\Q(6,q)$ in $ \PG(6,q)$, and not
just on its ``combinatorial properties''.  With the best new geometric construction, we obtain the
bound
\[c(q, 12) \le 2(q^5 - 3q^3)\]
for every prime power $q$, which improves the current best upper bound of $c(q, 12) \le 2(q^5 -
q^3)$ for all prime powers $q$.

\begin{rem}
For a prime power $q$, if $q - 1$ is also a prime power then we clearly have $c(k, g)$ equal to the
Moore bound.  Therefore, the upper bounds on $c(q,g)$ for $g \in \{6, 8, 12\}$ are interesting only
when $q - 1$ is not a prime power.
\end{rem}

Finding $1$-good structures in generalized quadrangles and hexagons of small order with the help of
a computer played a big role in obtaining these new constructions.  In Section \ref{sec:computer} we
describe our computational method and give a full computer-classification of $1$-good structures in
the known generalized quadrangles of orders $4$ and $5$, and a classification of $1$-good structures
that have a non-trivial automorphism group in the known generalized quadrangle of order $7$.

\section{Preliminaries}
\label{sec:prelim}

A \textit{generalized $n$-gon}, for $n \in \mathbb{N} \setminus \{0, 1, 2\}$, of order $(s, t)$ is a
point-line geometry $\mathcal{S} = (\mathcal{P}, \mathcal{L}, \inc)$ satisfying the following
properties:
\begin{enumerate}[(1)]
\item  every point of $\mathcal{S}$ is incident with exactly $t + 1$ lines;
\item  every line of $\mathcal{S}$ is incident with exactly $s + 1$ points;
\item the incidence graph of $\mathcal{S}$ has diameter $n$ and girth $2n$. 
\end{enumerate}
These objects were introduced by Tits \cite{Tits59} and a standard reference for them is \cite{VanMaldeghem98}. 

For a generalized polygon $(\mathcal{P}, \mathcal{L}, \inc)$ we will measure the distance $\dist(x,
y)$ between $x, y \in \mathcal{P} \cup \mathcal{L}$ by the distance between the vertices
corresponding to $x$ and $y$ in the incidence graph.  Two elements of the generalized $n$-gon that
have the same type are called \textit{opposite} if they are at distance $n$ from each other.

It was proved by Feit and Higman \cite{Feit-Higman} that finite generalized $n$-gons of order
$(s,t)$ with $s,t \ge 2$ can only exist for $n \in \{3, 4, 6\}.$ These generalized $n$-gons are
known as generalized triangles, quadrangles and hexagons, respectively. When $s = t \ge 2$, the
incidence graphs of these generalized polygons give us \textit{Moore graphs}, and such generalized
polygons are only known to exist when $s = t = q$ for a prime power $q$. As mentioned in the
introduction, we will be looking at regular induced subgraphs of the incidence graphs of these
generalized polygons.

\begin{mydef}
A $t$-good structure in a generalized polygon $(\mathcal{P}, \mathcal{L}, \inc)$  is a
pair of subsets $\overline{\mathcal{P}}\subseteq \mathcal{P}$, $\overline{\mathcal{L}} \subseteq
\mathcal{L}$ with the property that there are exactly $t$ lines of $\overline{\mathcal{L}}$ through
each point not in $\overline{\mathcal{P}}$, and there are exactly $t$ points of
$\overline{\mathcal{P}}$ on each line not in $\overline{\mathcal{L}} $.
\end{mydef}
Note that if the generalized polygon has order $q$ and $t \leq q$, then $(\overline{\mathcal{P}}, \overline{\mathcal{L}})$ is a $t$-good structure if and only if
the subgraph of the incidence graph induced by the points and lines not contained in
$\overline{\mathcal{P}} \cup \overline{\mathcal{L}}$ is $(q+1 - t)$-regular.  From the definition it
follows that we must have $|\overline{\mathcal{P}}| = |\overline{\mathcal{L}}| =
|\overline{\mathcal{P}} \cup \overline{\mathcal{L}}|/2$.  We call this quantity the \textit{size} of
the $t$-good structure. 

In this paper, we will be constructing new $1$-good structures in known generalized quadrangles and
hexagons of order $q$, for a prime power $q$, to improve the known upper bounds on $c(q,8)$ and
$c(q,12)$.  For these constructions, we will be using the symplectic generalized quadrangle
$\sympl(3,q)$ and the split Cayley generalized hexagon $\GH(q)$, which we describe below.

Let $\beta$ be a symplectic form defined on the three dimensional projective space $\PG(3,q)$, over
the finite field $\mathbb{F}_q$.  Since all symplectic forms on $\PG(3,q)$ are pairwise isometric,
we can take $\beta$ to be the form defined by $\beta((x_0, x_1, x_2, x_3), (y_0, y_1, y_2, y_3)) =
x_0y_1 - x_1y_0 + x_2y_3 - x_3y_1$.  Then the points and lines of $\PG(3,q)$ which are totally
isotropic with respect to $\beta$, that is, points $X$ which satisfy $\beta(X,X) = 0$ and lines $\ell$
for which all points $X, Y$ incident with $\ell$ satisfy $\beta(X,Y) = 0$, form a generalized
quadrangle of order $q$.  This generalized quadrangle is denoted by $\sympl(3,q)$ and it is known as
the finite \textit{symplectic generalized quadrangle}.  We note that the form $\beta$ defines a
polarity $\perp$ of $\PG(3,q)$ which maps a subspace $S$ to $S^\perp = \{Y : \beta(X,Y) =
0,\;\text{for all } X \in S\}$, and in fact the elements of $\sympl(3,q)$ are those elements $S$ of $\mathrm{PG}(3,q)$
which satisfy $S \subseteq S^\perp$. 

Now let $\Q$ be a non-singular quadric in $\PG(n,q)$, that is, the set of points satisfying an
irreducible quadratic form which cannot be described in fewer variables.  From the standard
classification of such quadrics, it follows that $\Q$ can be one of the following three types.
\begin{enumerate}[(1)]
\item $n = 2m$, $\Q$ is \textit{parabolic} with its quadratic form equivalent to \[\Q(x_0, \dots,
  x_{2m}) = x_0^2 + x_1x_2 + \cdots + x_{2m-1}x_{2m}.\]

\item $n = 2m - 1$, $\Q$ is \textit{hyperbolic} with its quadratic form equivalent to \[\Q(x_0,
  \dots, x_{2m-1}) = x_0x_1 + x_2x_3 + \cdots + x_{2m-2}x_{2m-1}.\]

\item $n = 2m - 1$, $\Q$ is \textit{elliptic} with its quadratic form equivalent to \[\Q(x_0, \dots,
  x_{2m-1}) = f(x_0, x_1) + x_2x_3 + \cdots + x_{2m-2}x_{2m-1},\] with $f(x_0, x_1)$ an irreducible
  degree $2$ polynomial over $\mathbb{F}_q$.
\end{enumerate}
The parabolic quadrics are denoted by $\Q(2m,q)$, the hyperbolic by $\Q^+(2m-1,q)$ and the elliptic
by $\Q^-(2m-1,q)$.  
Let $\Q$ be a non-singular quadric. 
A \textit{cone} with vertex a point $X$ over $\Q$, denoted by $X \cdot \Q$, is the set of all points lying on the lines joining $X$ to a
point of a quadric isomorphic to $\Q$ lying in a hyperplane which  does not contain $X$.  Similarly, a cone with vertex a
line $\ell$ over $\Q$, denoted by $\ell \cdot \Q$, is the set of all points lying on the lines joining the points of
$\ell$ to the points of a quadric $\Q$ in an $(n-2)$-dimensional subspace disjoint from $\ell$.

A subspace $S$ of $\PG(n,q)$ is called \textit{totally singular} with respect to a quadric $\Q$ if
all of its points are contained in $\Q$.  The maximum \textit{vector space dimension} of a totally
singular subspace of $\Q$ is called the \textit{Witt index} of $\Q$, and the Witt indices of
$\Q(2m,q)$, $\Q^+(2m-1,q)$ and $\Q^-(2m-1,q)$ are $m$, $m$ and $m - 1$, respectively.  If the Witt
index of a quadric $\Q$ is $2$, then its totally singular points and lines form a generalized
quadrangle.  The quadric $\Q(4,q)$ gives rise the to the generalized quadrangle which is the
point-line dual of $\sympl(3,q)$, and therefore it has the same incidence graph as $\sympl(3,q)$.

A standard way of constructing the split Cayley generalized hexagon of order $q$, denoted by
$\GH(q)$, is by using the parabolic quadric $\Q(6,q)$ in the following way.  The points of $\GH(q)$
are all the points of $\Q(6,q)$ while the lines of $\GH(q)$ are only those lines of
$\Q(6,q)$ whose Grassmann coordinates satisfy a certain condition.  Since we will not be using this
condition on lines directly, we refer the interested reader to \cite[Section
  2.4.13]{VanMaldeghem98}.  What we will need is some of the well known geometric properties of
$\GH(q)$ which follow from its definition.  We summarise these properties and introduce some
terminology.

\begin{enumerate}[P1.]
\item The set of points of $\GH(q)$ is identical to the set of points of $\Q(6,q)$.
\item The set of $q+1$ $\GH(q)$-lines incident with a point $P$ span a totally singular plane of
  $\Q(6,q)$, called an \emph{$\GH(q)$-plane} with \emph{centre} $P$.
\item Every plane of $\Q(6,q)$ is either an $\GH(q)$-plane or contains no $\GH(q)$-lines. In the
  latter case the plane is called an \emph{ideal plane} of $\Q(6,q)$.
\item For a point $P$, the set of points at distance at most $4$ from $P$ in $\GH(q)$ is equal to
  the set of points collinear with $P$ in the quadric $\Q(6,q)$.  This set is denoted by $P^\perp$.
\item Every line of $\Q(6, q)$ is incident with exactly $q + 1$ planes of $\Q(6, q)$, and if the line is an $\GH(q)$-line, then each of these planes is an
$\GH(q)$-plane. 
\item Every \emph{ideal line}, that is, a line of $\Q(6,q)$ not in $\GH(q)$, is incident with a
  unique $\GH(q)$-plane.
\end{enumerate}

Proofs of these properties can be found in \cite[Section 1.4.2]{Offer_thesis}.  Finally, we refer
the reader to \cite{Ball_Book} for a quick introduction to the basic notions from finite geometry
and to \cite{Hirschfeld-Thas_book} for a standard reference on the subject.

\section{Bounds on  regular induced subgraphs}
\label{sec:bounds}
Recall that a $t$-good structure in a generalized polygon is equivalent to the collection of points
and lines not contained in a $(q + 1 - t)$-regular induced subgraph of the incidence graph.
Therefore, to study how big a $t$-good structure can be, we will give lower bounds on the size of
a regular induced subgraph.  First we recall the Expander Mixing Lemma for bipartite graphs
\cite[Section 2.4]{Hoory-Linial-Wigderson06} (one of the oldest references for this lemma is Theorem
3.1.1 in \cite{Haemers_thesis}, which is also Theorem 5.1 in \cite{Haemers95_interlacing}).  A
direct proof of this lemma is given in \cite[Section 3.2]{WSV12}.

\begin{lem}
\label{lem:expander_mixing}
Let $G = (L, R, E)$ be a biregular bipartite graph and let $\lambda_1 \ge \lambda_2 \ge \dots \ge
\lambda_{|L| + |R|}$ be the eigenvalues of its adjacency matrix.  Let $S \subseteq L$ and $T
\subseteq R$ be such that $|S| = \alpha |L|$ and $|T| = \beta |R|$ for some real numbers $\alpha, \beta
\in [0, 1]$.  Denote by $e(S, T)$ the number of edges which have vertices in the sets $S$ and $T$.
Then we have
\[\left\lvert \frac{e(S,T)}{|E|} - \alpha \beta \right\rvert \le \frac{\lambda_2}{\lambda_1} \sqrt{\alpha \beta (1 - \alpha)(1 - \beta)}.\]
\end{lem}

\begin{thm}
\label{thm:main}
Let $G = (L, R, E)$ be a $d$-regular bipartite graph and let $\lambda$ be its second largest
eigenvalue.  Let $H$ be a $k$-regular induced subgraph of $G$, with $k \geq 1$.  Then we have
\[\frac{k - \lambda}{d - \lambda} \le \frac{|V(H)|}{|V(G)|} \le \frac{k + \lambda}{d + \lambda}.\]
\end{thm}
\begin{proof}
The subgraph $H$ must have equally many vertices in both $L$ and $R$ since it is $k$-regular.  Let
$x = |L \cap V(H)| = |R \cap V(H)| = |V(H)|/2$.  We must also have $|L| = |R|$ since $G$ is
$d$-regular. Let $n = |L| = |R| = |V(G)|/2$.  Then by applying Lemma \ref{lem:expander_mixing} to
the sets $L \cap V(H)$ and $R \cap V(H)$, both of size $x$, which have $kx$ edges between them we
get the following:
\[\left\lvert \frac{kx}{dn} - \frac{x^2}{n^2} \right\rvert \le \frac{\lambda}{d} \sqrt{x^2\left(1 - \frac{x}{n}\right)^2}.\]
Simplifying this we get the inequality
\[\frac{k - \lambda}{d - \lambda}  \le \frac{x}{n} \le \frac{k + \lambda}{d + \lambda}.\qedhere \] 
\end{proof}

\begin{rem}
In our paper we will only be interested in the lower bound in Theorem \ref{thm:main}.  Note that
this lower bound is well defined only for $d > \lambda$ and positive only for $k > \lambda$. 
\end{rem}

\begin{rem}
The upper bound in Theorem \ref{thm:main} can be improved to $(k + |\lambda_n|)/(d + |\lambda_n|)$
when $G$ is a non-bipartite graph and $k$ is allowed to be $0$, in which case we get a generalization of the well known
Hoffman-Delsarte bound on independent sets in a regular subgraph since an independent set is
equivalent to an induced $0$-regular subgraph.  This generalization, along with Theorem
\ref{thm:main}, was already proved by Haemers in \cite[Theorem 2.1.4]{Haemers_thesis}.
\end{rem}

We are now ready to prove our main result regarding the size of $t$-good structures in generalized
polygons.

\begin{thm}
\label{thm:main_tgood}
Let $(\overline{\mathcal{P}}, \overline{\mathcal{L}})$ be a $t$-good structure in a generalized
$n$-gon of order $q$ for integers $q \ge 2$ and $q \geq t \geq 1$.  Then
\[|\overline{\mathcal{P}}| = |\overline{\mathcal{L}}| \le 
\begin{cases}
t(q + \sqrt{q} + 1), & \text{if } n = 3;\\
t(q + 1)(q + \sqrt{2q} + 1), & \text{if } n = 4; \\
t(q+1)(q^2 + 1)(q + \sqrt{3q} + 1), & \text{if } n = 6.
\end{cases}\]
\end{thm}
\begin{proof}
Let $H$ be the $(q + 1 - t)$-regular induced subgraph of the incidence graph $G$ of the generalized
$n$-gon whose vertices are the points and lines not in $\overline{\mathcal{P}} \cup
\overline{\mathcal{L}}$.  The lower bound in Theorem \ref{thm:main} can be rephrased as
\begin{equation}
\label{eq}
1 - \frac{|V(H)|}{|V(G)|} \le \frac{q + 1 - (q + 1 - t)}{q + 1 - \lambda} = \frac{t}{q + 1 - \lambda}.
\end{equation}
If $\theta_n$ is the total number of points in the generalized $n$-gon, then the left hand side in
(\ref{eq}) is equal to $|\overline{\mathcal{P}}|/\theta_n$.  It is well known, and easy to prove,
that the second largest eigenvalue of a generalized $n$-gon of order $q$ is $\sqrt{q}, \sqrt{2q},
\sqrt{3q}$ for $n = 3$, $4$ and $6$, respectively \footnote{See for example the proof of Proposition
  7.2.7 in \cite{VanMaldeghem98}. For our purposes, we only need these values as upper bounds to the
  second largest eigenvalues, which is proved in Section 3.4 of \cite{WSV12}.}.  The number of
points $\theta_n$ in a generalized $n$-gon of order $q$ is equal to $q^2 + q + 1$, $(q + 1)(q^2
+ 1)$ and $(q + 1)(q^4 + q^2 + 1)$, for $n = 3$, $4$ and $6$, respectively.  We  now substitute these
values in (\ref{eq}) for each $n \in \{3, 4, 6\}$ and simplify the expression by noting that $q^2 + 1 = (q - \sqrt{2q} + 1)(q + \sqrt{2q} + 1)$ and $q^4 + q^2 + 1 = (q^2 + 1)(q - \sqrt{3q} + 1)(q + \sqrt{3q} + 1)$. 
\end{proof}

Note that for $t = 1$ and $n = 4$, our bound improves the Beukemann-Metsch bound of $2q^2 + 2q - 1$
on the number of points of a $1$-good structure in a generalized quadrangle of order $q$
\cite[Theorem 1.1]{BM11}. In fact, Beukemann and Metsch proved the bound only for the quadrangle
$\Q(4,q)$ while our bound holds for arbitrary generalized quadrangles of order $q$, where $q$ can be
any integer.  Our bound also answers the question of Metsch \cite[Section 6]{Metsch11} by proving
that there cannot exist any constant $c > 1$ for which the parabolic quadric $\Q(4,q)$ has a
$1$-good structure of size greater than $c q^2$.

\section{Constructions in generalized quadrangles}
\label{sec:const}
In this section we give a construction that ``lifts'' a $1$-good structure in $\PG(2,q)$ to a
$1$-good structure in the symplectic generalized quadrangle $\sympl(3,q)$ (see \cite[Chapter 3]{FGQ} for a description of this object).  By \cite[Theorem
  3.10]{Gacs-Heger08}, there are only three kinds of $1$-good structures in $\PG(2, q)$ and they are
of sizes $q + 1$, $q + 2$ and  $q + \sqrt{q} + 1$, respectively; the last possibility corresponds to
a Baer subplane in $\PG(2,q)$ which only exists when $q$ is a square.  If we start with a Baer
subplane of $\PG(2,q)$ in our construction, then we obtain a $1$-good structure which is larger than
those constructed by Beukemann and Metsch \cite{BM11}.

\begin{thm}
\label{thm:construction}
Let $\sympl(3,q)$ be the generalized quadrangle obtained by taking a symplectic form on $\PG(3,q)$
and let $\perp$ be the polarity defined by this form.  Denote the set of points and lines of
$\sympl(3,q)$ by $\m P$ and $\m L$, respectively.  Fix a point $P$ of $\sympl(3,q)$ and let $\pi =
P^\perp$ be the plane in $\PG(3,q)$ which contains all points of $\sympl(3,q)$ collinear with $P$.
Let $(\overline{\mathcal{P}}', \overline{\mathcal{L}}')$ be a $1$-good structure in $\pi \cong
\PG(2,q)$.  Define the following sets in $\sympl(3,q)$.
\begin{itemize}
\item $\overline{\mathcal{P}} = \{P\} \cup \overline{\mathcal{P}}' \cup \{X \in \m P
  \mid X^\perp \cap \pi \in \overline{\mathcal{L}}'\}$,
\item $\overline{\mathcal{L}} = \{\ell \in \m L \mid P \inc \ell\} \cup \{\ell \in \m L \setminus \pi
  \mid \ell \cap \pi \in \overline{\mathcal{P}}'\}$. 
\end{itemize}
Then $(\overline{\mathcal{P}}, \overline{\mathcal{L}})$ is a $1$-good structure in $\sympl(3,q)$.
Moreover, if $P \in \overline{\mathcal{P}}'$ then $|\overline{\mathcal{P}}| =
q|\overline{\mathcal{P}}'| + 1$ and if $P \not \in \overline{\mathcal{P}}'$, then
$|\overline{\mathcal{P}}| = q|\overline{\mathcal{P}}'| + q + 1$.
\end{thm}
\begin{proof}
Let $\ell$ be a line in $\m L \setminus \overline{\mathcal{L}}$.  Then $\ell$ is not contained in
$\pi$ (because it is not incident with $P$) and the point $Q = \ell \cap \pi$ is not contained in
$\overline{\mathcal{P}}'$.  Therefore there exists a unique line $m$ of the plane $\pi$ (which may
or may not be in the set $\m L$) through $Q$ that lies in $\overline{\mathcal{L}}'$ since
$(\overline{\mathcal{P}}', \overline{\mathcal{L}}')$ is a $1$-good structure in $\pi$.  There is a
bijective correspondence between the $q + 1$ points $Y$ on $\ell$ and the $q + 1$ lines through $Q$
in $\pi$ via the map $Y \mapsto Y^\perp \cap \pi$.  Then the unique point on $\ell$ corresponding to
$m$ is the unique point of $\overline{\mathcal{P}}$ incident with $\ell$.

Now let $X$ be a point in $\m P \setminus \overline{\mathcal{P}}$.  If $X \in \pi$, then the line
$PX$ lies in $\overline{\mathcal{L}}$ while every other line of $\sympl(3,q)$ through $X$, which
must lie outside the plane $\pi$, is in $\m L \setminus \overline{\mathcal{L}}$.  Now say $X \not
\in \pi$.  Then the line $X^\perp \cap \pi$ is a line of $\pi$ not contained in
$\overline{\mathcal{L}}'$, and hence it contains a unique point $Y$ of $\overline{\mathcal{P}}$.
The line $XY$ is the unique line of $\overline{\mathcal{L}}$ through $X$.  Therefore,
$(\overline{\mathcal{P}}, \overline{\mathcal{L}})$ is a $1$-good structure in $\sympl(3,q)$.

Note that $|\overline{\mathcal{L}}| = |\overline{\mathcal{P}}|$ and $|\overline{\mathcal{L}}'| =
|\overline{\mathcal{P}}'|$, and thus it suffices to calculate $|\overline{\mathcal{L}}|$.  There are
$q + 1$ lines of $\overline{\mathcal{L}}$ in $\pi$.  And for every point $X \in
\overline{\mathcal{P}}' \setminus \{P\}$, every line of $\sympl(3,q)$ through $X$ is in
$\overline{\mathcal{L}}$.  Therefore,
\[|\overline{\mathcal{P}}| = |\overline{\mathcal{L}}| = 
\begin{cases}
q + 1 + q(|\overline{\mathcal{P}}'| - 1), & \text{ if } P \in \overline{\mathcal{P}}'\\
q + 1 + q|\overline{\mathcal{P}}'| , & \text{ if } P \not \in \overline{\mathcal{P}}'.
\end{cases}\qedhere
\]
\end{proof}

From Theorem \ref{thm:construction} and \cite[Theorem 3.10]{Gacs-Heger08}, we can construct $1$-good
structures of sizes $q^2 + q + 1, q^2 + 2q + 1, q^2 + 3q + 1$, $q^2 + q\sqrt{q} + q + 1$ and $q^2 +
q\sqrt{q} + 2q + 1$ in a generalized quadrangle of order $q$ by starting with a $1$-good structure
of size $q + 1$, $q + 2$ or $q + \sqrt{q} + 1$ in the plane $\pi$.  Previously, the best known
construction was of size $q^2 + 3q + 1$ for $q$ odd and $q^2 + 4q + 3$ for $q$ even \cite{BM11}.
Therefore, for $q$ square, we have constructed larger $1$-good structures and proved the following.

\begin{thm}
For $q$ an even power of a prime, we have
\[c(q, 8) \le 2(q^3 - q\sqrt{q} - q).\]
\end{thm}

\begin{rem}
Unlike the Desarguesian projective planes, the problem of classifying $1$-good structures in the
classical generalized quadrangles appears to be extremely hard.
\end{rem}

\section{Constructions in split Cayley hexagons}
\label{sec:const2}

In this section we construct a large family of $1$-good structures in $\GH(q)$, which includes the
two known constructions in generalized hexagons \cite[Section 2]{Gacs-Heger08}.  The largest
$1$-good structure that we attain from our construction is of size $q^4 + 4q^3 + q^2 + q + 1$, which
implies the bound $c(q, 12) \le 2(q^5 - 3q^3)$ for all prime powers $q$. 

\begin{thm}
\label{1goodHq}
Let $q$ be a prime power and let $\GH(q) = (\mathcal{P}, \mathcal{L}, \inc)$ be the split Cayley hexagon
in its standard representation inside the quadric $\Q(6,q)$ in $\PG(6,q)$.  Let $S$ be a $4$-space
in $\PG(6,q)$ and let $A$ be a point in $S \cap \Q(6,q)$.  Let $\pi_A$ denote the $\GH(q)$-plane on
$A$ consisting of the $q+1$ lines of $\mathcal{L}$ containing $A$.  Define the following sets of
points and lines in $\GH(q)$:
\begin{align*}
\mathcal{P}_1&=\{ X\in \mathcal{P}  \colon \dist(X,A)\le 4\} &\mathcal{L}_1&=\{n\in\mathcal{L}\colon n\cap S\ne \varnothing \}\\
\mathcal{P}_2&=\{X\in\mathcal{P}\colon X\inc S\} &\mathcal{L}_2&=\{n\in\mathcal{L}\colon n\cap \pi_A\ne \varnothing \}\\
\mathcal{P}_3&=  \{X\in \mathcal{P}  \colon X\notinc S, \; | \Gamma_2(X)\cap S|\ne 1\}.
\end{align*} 

Then $\overline{\mathcal{P}} = \mathcal{P}_1 \cup \mathcal{P}_2 \cup \mathcal{P}_3$ and
$\overline{\mathcal{L}} = \mathcal{L}_1 \cup \mathcal{L}_2$ form a $1$-good structure.
\end{thm}
\begin{proof}
We split the proof into various cases.
\begin{description}
\item[Elements of $\mathcal{P}\backslash\overline{\mathcal{P}}$ are incident with $0$ elements of
  $\mathcal{L}_2$]

Let $X$ be a point incident with a line $n$ of $\mathcal{L}_2$. Now $n\cap\pi_A\ne
\varnothing$, which is equivalent to $\dist(n,A)\le 3$.  Moreover, since $X$ is a point of $n$, we
have $\dist(X,A)\le 4$.  Therefore, every point not lying in $\mathcal{P}_1$ is incident with $0$
elements of $\mathcal{L}_2$.

\item[Elements of $\mathcal{P}\backslash\overline{\mathcal{P}}$ are incident with 1 element of
  $\mathcal{L}_1$]

Let $X$ be a point not in $\overline{\mathcal{P}}$.  Then $X$ must be a point outside $S$ which is
collinear with a unique point in $S$, since it does not lie in $\mathcal{P}_2 \cup \mathcal{P}_3$.
Therefore, there is a unique line $n$ through $X$ containing a point of $S$.  This line $n$ must lie
in $\mathcal{L}_1$.

\item[Elements of $\mathcal{L}\backslash\overline{\mathcal{L}}$ are incident with 1 element of
  $\mathcal{P}_1$] Every line not lying in the tangent hyperplane $T_A$ through $A$ meets $T_A$ in
  precisely one point. Also, a line of $\GH(q)$ lying in $T_A$ must necessarily meet the $\GH(q)$-plane $\pi_A$.
  So every line not lying in $\mathcal{L}_1\cup \mathcal{L}_2$
  meets $\mathcal{P}_1$ in precisely one point.

\item[Elements of $\mathcal{L}\backslash\overline{\mathcal{L}}$ are incident with 0 elements of
  $\mathcal{P}_2$] Every line which is incident with more than $0$ elements of $\mathcal{P}_2 = S$
  is already contained in $\mathcal{L}_1$, therefore lines not contained in $\mathcal{L}_1 \cup
  \mathcal{L}_2$ are incident with $0$ elements of $S$.

\item[Elements of $\mathcal{L}\backslash\overline{\mathcal{L}}$ are incident with 0 elements of
  $\mathcal{P}_3$] Let $\ell$ be a line not contained in $\mathcal{L}_1$.  For the sake of
  contradiction, assume that $\ell$ contains a point $X$ of $\mathcal{P}_3$.  
  Note that the $\GH(q)$-plane $\pi_X$ must meet $S$, and so $|\Gamma_2(X)\cap S|\ne 0$.
  Since there are at
  least two points collinear with $X$ in $\GH(q)$ which are contained in $S$, the $\GH(q)$-plane
  $\pi_X$ must intersect $S$ in a line $m$.  The lines $m$ and $\ell$ lie in $\pi_X$ and so must
  intersect in a point of $S$ incident with $\ell$, a contradiction.  \qedhere
\end{description}
\end{proof}

For the rest of this section, we assume that $\overline{\mathcal{P}} = \mathcal{P}_1 \cup
\mathcal{P}_2 \cup \mathcal{P}_3, \overline{\mathcal{L}} = \mathcal{L}_1 \cup \mathcal{L}_2$ is a
$1$-good structure in $\GH(q)$ as defined in Theorem \ref{1goodHq}, and we use all the notation
defined there.  To find out the size of this $1$-good structure we will be determining
$|\mathcal{L}_1|$, $|\mathcal{L}_2|$ and $|\mathcal{L}_1 \cap \mathcal{L}_2|$ for different cases
that we describe below.  While an easy count will show that $|\mathcal{L}_2|$ is always equal to
$q^3 + q^2 + q + 1$, the other two quantities depend on how the $4$-space $S$ intersects $\Q(6,q)$
and $\GH(q)$, and where the point $A$ is located with respect to this intersection.   
These intersections have been studied before (see for example \cite[Lemma 4.1]{Ihringer:2014aa}), and the following result is probably 
known but we prove it here for completeness. 
\begin{lem}
\label{lem:4dim_int}
Any $4$-space $S$ in $\PG(6,q)$ intersects $\GH(q)$ inside $\Q(6,q)$ in exactly one of the following
ways.
\begin{enumerate}[$(a)$]
\item $S \cap \Q(6,q) \cong \Q(4,q)$: there are $q+1$ $\GH(q)$-lines contained in $S$ and they are
  pairwise opposite to each other.

\item $S \cap \Q(6,q) \cong P\cdot \Q^-(3,q)$: there is a unique $\GH(q)$-line $\ell$ through $P$ in
  $S$.

\item $S \cap \Q(6,q) \cong P\cdot \Q^+(3,q)$:
\begin{enumerate}[$(i)$]
\item $\pi_P \inc S$: $S$ meets $\GH(q)$ in $q+1$ $\GH(q)$-planes, one through each $\GH(q)$-line
  containing $P$, and hence a total of $(q+1)^2$ lines of $\GH(q)$ are contained in $S$.

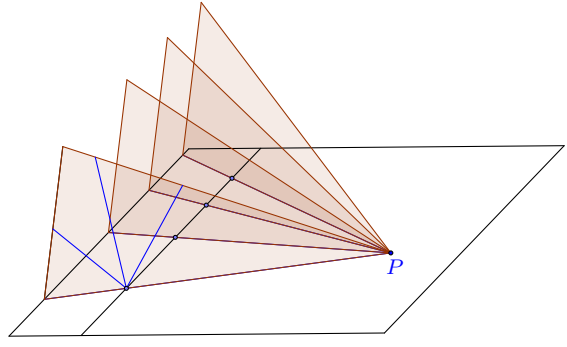
\definecolor{zzttqq}{rgb}{0.6,0.2,0.}
\definecolor{xdxdff}{rgb}{0.49019607843137253,0.49019607843137253,1.}
\definecolor{qqqqff}{rgb}{0.,0.,1.}
\begin{center}
\begin{tikzpicture}[line cap=round,line join=round,>=triangle 45,x=1.0cm,y=1.0cm, scale = 0.3]
\clip(-4.3,-9.8) rectangle (28.2,7.5);
\fill[color=zzttqq,fill=zzttqq,fill opacity=0.1] (0.9297875586880295,-5.991543022089227) -- (1.74,0.76) -- (16.12,-3.94) -- cycle;
\fill[color=zzttqq,fill=zzttqq,fill opacity=0.1] (3.7399919964209927,-3.038688612897738) -- (4.550204437732964,3.7128544091914897) -- (16.12,-3.94) -- cycle;
\fill[color=zzttqq,fill=zzttqq,fill opacity=0.1] (5.516711351057874,-1.1717804585330969) -- (6.326923792369845,5.579762563556129) -- (16.12,-3.94) -- cycle;
\fill[color=zzttqq,fill=zzttqq,fill opacity=0.1] (6.997489941397164,0.38416455771174013) -- (7.807702382709133,7.135707579800966) -- (16.12,-3.94) -- cycle;
\draw (7.26,0.66)-- (-0.62,-7.62);
\draw (-0.62,-7.62)-- (15.84,-7.48);
\draw (7.26,0.66)-- (23.72,0.8);
\draw (23.72,0.8)-- (15.84,-7.48);
\draw (2.5404518009147035,-7.5931188789715645)-- (10.420451800914703,0.6868811210284362);
\draw [color=qqqqff] (0.9297875586880295,-5.991543022089227)-- (16.12,-3.94);
\draw [color=qqqqff] (16.12,-3.94)-- (3.7399919964209927,-3.038688612897738);
\draw [color=qqqqff] (5.516711351057874,-1.1717804585330969)-- (16.12,-3.94);
\draw [color=qqqqff] (16.12,-3.94)-- (6.997489941397164,0.38416455771174013);
\draw (1.74,0.76)-- (0.9297875586880295,-5.991543022089227);
\draw [color=zzttqq] (0.9297875586880295,-5.991543022089227)-- (1.74,0.76);
\draw [color=zzttqq] (1.74,0.76)-- (16.12,-3.94);
\draw [color=zzttqq] (16.12,-3.94)-- (0.9297875586880295,-5.991543022089227);
\draw [color=zzttqq] (3.7399919964209927,-3.038688612897738)-- (4.550204437732964,3.7128544091914897);
\draw [color=zzttqq] (4.550204437732964,3.7128544091914897)-- (16.12,-3.94);
\draw [color=zzttqq] (16.12,-3.94)-- (3.7399919964209927,-3.038688612897738);
\draw [color=zzttqq] (5.516711351057874,-1.1717804585330969)-- (6.326923792369845,5.579762563556129);
\draw [color=zzttqq] (6.326923792369845,5.579762563556129)-- (16.12,-3.94);
\draw [color=zzttqq] (16.12,-3.94)-- (5.516711351057874,-1.1717804585330969);
\draw [color=zzttqq] (6.997489941397164,0.38416455771174013)-- (7.807702382709133,7.135707579800966);
\draw [color=zzttqq] (7.807702382709133,7.135707579800966)-- (16.12,-3.94);
\draw [color=zzttqq] (16.12,-3.94)-- (6.997489941397164,0.38416455771174013);
\draw [color=qqqqff] (4.527017487361625,-5.50571229473546)-- (1.3034666624189748,-2.8776553345465734);
\draw [color=qqqqff] (4.527017487361625,-5.50571229473546)-- (3.162048757959804,0.29521354920646187);
\draw [color=qqqqff] (4.527017487361625,-5.50571229473546)-- (6.992798920042276,-0.9568397026563764);
\begin{scriptsize}
\draw [fill=qqqqff] (16.12,-3.94) circle (2.5pt);
\draw[color=qqqqff] (16.3,-4.53) node {$P$};
\draw[color=black] (7.64,-2.91) node {}; 
\draw [fill=xdxdff] (4.527017487361625,-5.50571229473546) circle (2.5pt);
\draw [fill=xdxdff] (6.671730803876847,-3.2521302819504285) circle (2.5pt);
\draw [fill=xdxdff] (8.027701070109181,-1.8273290377367575) circle (2.5pt);
\draw [fill=xdxdff] (9.157813000354839,-0.6398510196613696) circle (2.5pt);
\end{scriptsize}
\end{tikzpicture}
\end{center}

\item $\pi_P \notinc S$: $S$ meets $\GH(q)$ in two collinear $\GH(q)$-planes, and hence $2q + 1$
  lines of $\GH(q)$ are contained in $S$.

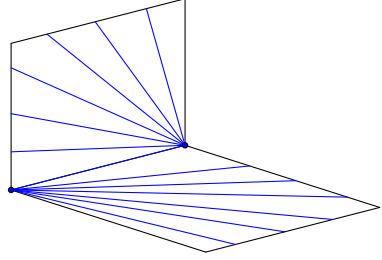
\begin{center}
\definecolor{qqqqff}{rgb}{0.,0.,1.}
\begin{tikzpicture}[line cap=round,line join=round,scale=0.4]
\clip(0.54,-2.6) rectangle (13.12,6.3);
\draw (0.7,-0.26)-- (7.1,-2.32);
\draw (7.1,-2.32)-- (12.82,-0.84);
\draw (0.7,-0.26)-- (0.7,4.6);
\draw (0.7,4.6)-- (6.42,6.08);
\draw (0.7,-0.26)-- (6.42,1.22);
\draw (6.42,1.22)-- (6.42,6.08);
\draw (6.42,1.22)-- (12.82,-0.84);
\draw [color=qqqqff] (0.7,-0.26)-- (6.42,1.22);
\draw [color=qqqqff] (0.7,-0.26)-- (8.544953615292021,0.5360305550778809);
\draw [color=qqqqff] (0.7,-0.26)-- (10.06210735910362,0.04769669378852215);
\draw [color=qqqqff] (0.7,-0.26)-- (8.080109744472088,-2.066405171010718);
\draw [color=qqqqff] (0.7,-0.26)-- (9.701722895207979,-1.6468269431979352);
\draw [color=qqqqff] (0.7,-0.26)-- (11.27712013764561,-1.239206677672115);
\draw [color=qqqqff] (0.7,-0.26)-- (11.759927051561693,-0.49878901972141954);
\draw [color=qqqqff] (6.42,1.22)-- (5.14736683019833,5.750717291729638);
\draw [color=qqqqff] (6.42,1.22)-- (3.465972238368714,5.315671138598898);
\draw [color=qqqqff] (6.42,1.22)-- (1.8803343486247064,4.9054011950987);
\draw [color=qqqqff] (6.42,1.22)-- (0.7,3.7898795539033463);
\draw [color=qqqqff] (6.42,1.22)-- (0.7,2.26888624535316);
\draw [color=qqqqff] (6.42,1.22)-- (0.7,1.0039613382899626);
\begin{scriptsize}
\draw [fill=qqqqff] (0.7,-0.26) circle (2.5pt);
\draw [fill=qqqqff] (6.42,1.22) circle (2.5pt);
\end{scriptsize}
\end{tikzpicture}
\end{center}
\end{enumerate}

\item $S \cap \Q(6, q) \cong m \cdot \Q(2,q)$: 
\begin{enumerate}[$(i)$]
\item $m$ does not lie in $\GH(q)$ and the lines of $\GH(q)$ contained in $S$ are precisely the
  $q+1$ lines on a point $P$, and $m\inc \pi_P$ (but $m\notinc P$).

\item $m$ lies in $\GH(q)$ and the lines of $\GH(q)$ contained in $S$ are precisely the $q^2+q+1$
  lines concurrent with or equal to $m$.
\end{enumerate}

\end{enumerate}

\end{lem}

\begin{proof}\leavevmode
\begin{enumerate}[$(a)$]
\item Suppose $S$ is non-degenerate.  If $q$ is odd, then the perp of $S$ with respect to the
  polarity for $\Q(6,q)$ is a line $m$ of plus type.  So the set of $\GH(q)$-lines within $S$ form
  what is often called the \emph{distance-3 trace} of the two points of $m \cap \Q(6,q)$, and so by
  \cite[Lemma 1.17]{Offer_thesis}, the $\GH(q)$-lines within $S$ form a regulus of a
  $\Q^+(3,q)$-section of $\Q(6,q)$.  For $q$ even, the pole of $S$ with respect to the quadric
  $\Q(6,q)$ is a plane of $\PG(6,q)$ incident with the nucleus $\eta$ of $\Q(6,q)$, meeting
  $\Q(6,q)$ in a conic $\mathcal{C}$.  Any two points of $\mathcal{C}$ are opposite (in $\GH(q)$)
  and so the distance-$3$-trace of two points of $\mathcal{C}$ yield a regulus of a
  $\Q^+(3,q)$-section of $\Q(6,q)$ (see \cite[Lemma 1.17]{Offer_thesis}). Moreover, any point of $S$
  is incident with at most one line of $\GH(q)$, since two concurrent $\GH(q)$-lines span a totally
  singular plane; but $\Q(4,q)$ does not contain any totally singular planes. Therefore, any two
  points of $\mathcal{C}$ will yield the same regulus of $q+1$ $\GH(q)$-lines of $S$.
 
\item Suppose $S\cap \Q(6,q)=P\cdot \Q^-(3,q)$. Therefore, every $\GH(q)$-line contained in $S$ must
  be a generator of the cone $\mathcal{C}=P\cdot \Q^-(3,q)$.  However, each such $\GH(q)$-line is
  incident with $P$, and the quotient polar space $\Q^-(3,q)$ does not contain any totally isotropic
  lines. Therefore, there can be at most one $\GH(q)$-line of $S$ incident with $P$ since two such
  lines would span a totally isotropic plane.  Now the unique $\GH(q)$-plane $\pi_P$ incident with
  $P$ is not contained in $S$ (by the above argument) and so meets $S$ in a line $\ell$, and $\ell$
  is an $\GH(q)$-line. Therefore, there is a unique $\GH(q)$-line within $S$, and it is incident
  with $P$.

\item Suppose $S\cap \Q(6,q)=P\cdot \Q^+(3,q)$. We have two subcases:
\begin{enumerate}[(i)]
\item $\pi_P\inc S$: In this case, each $\GH(q)$-line incident with $P$ is contained in $S$.  Now
  $S/P$ is isometric to $\Q^+(3,q)$ and so for each $\GH(q)$-line $z$ incident with $P$, there are
  two totally isotropic planes incident with $z$ and contained in $S$; one of which is the
  $\GH(q)$-plane $\pi_P$ with centre $P$. Since $z$ lies in $\GH(q)$, it follows that the two
  totally isotropic planes on $z$ are $\GH(q)$-planes.

\item $\pi_P\notinc S$: So $\pi_P$ meets $S$ in an $\GH(q)$-line $\ell$ on $P$, since both $S$ and
  $\pi_P$ lie in the tangent hyperplane $T_P$ to $P$. Now $\ell$ is incident with $S$ and contains
  the radical $P$ of $S$. So the quotient space $S/P$ is non-degenerate and isomorphic to $\Q^+(3,q)$
  and the image $\ell/P$ of $\ell$ is a point of $\Q^+(3,q)$. There are precisely two totally
  singular lines of $S/P$ ($\cong\Q^+(3,q)$) incident with $\ell/P$, which implies that there are
  precisely two totally isotropic planes incident with both $\ell$ and $S$. Since $\ell\in\GH(q)$,
  we have shown that there are precisely two $\GH(q)$-planes incident with $\ell$ and contained in
  $S$. The result then follows.
\end{enumerate}

\item Suppose $S \cap \Q(6, q) \cong m \cdot \Q(2,q)$. We have two subcases:
\begin{enumerate}[$(i)$]
\item $m$ does not lie in $\GH(q)$: Suppose $X$ is a point of $m$. Then there cannot be two
  $\GH(q)$-lines incident with $X$ in $S$ since otherwise, $X$ would be the centre of an
  $\GH(q)$-plane $\pi_X$, and $\pi_X$ would have to contain $m$ since it would be incident with
  $S$. However, this would imply that $m$ is a line of $\GH(q)$; a contradiction. On the other hand,
  the $\GH(q)$-plane $\pi_X$ incident with $X$ meets $S$ in a line, which is necessarily in
  $\GH(q)$. So there is a unique $\GH(q)$-line on each point of $m$. Moreover, since $S$ meets
  $\Q(6,q)$ in a cone of type $m\cdot \Q(2,q)$, it is clear that we have accounted for all of the
  $\GH(q)$-lines of $S$, and there are $q+1$ of them. Let $\pi$ be the span of these lines.  Then
  $m$ and $S$ are both incident with $\pi$, and hence $\pi$ is totally isotropic. Therefore, $\pi$
  is a plane and so must be an $\GH(q)$-plane (as it contains $\GH(q)$-lines).

\item $m$ lies in $\GH(q)$: This part follows from the basic properties of $\GH(q)$; every line of
  $\GH(q)$ concurrent with or equal to $m$ lies in $S$ and vice-versa.\qedhere
\end{enumerate}
\end{enumerate}
\end{proof}

Since there are $q+1$ lines of $\GH(q)$ incident with the $\GH(q)$-plane $\pi_A$ incident with $A$,
and each point of $\pi_A$, apart from $A$, is incident with $q$ lines of $\GH(q)$ not incident with
$\pi_A$, it follows that
\[
|\mathcal{L}_2|=q+1+(q^2+q)q=q^3+q^2+q+1.
\]
We now calculate the cardinality of $\mathcal{L}_1$.

\begin{lem}
\label{lem:size_L1}
\[
|\mathcal{L}_1|=
\begin{cases}
q^4 +  2q^3 + q^2 + q + 1 & S \cap \Q(6,q) \cong \Q(4,q),\\
q^4+q^3+q^2+q+1 & S \cap \Q(6,q) \cong P \cdot \Q^-(3,q), \\
q^4+2q^3+q^2+q+1 & S \cap \Q(6,q) \cong P \cdot \Q^+(3,q) \text{ and } \pi_P \inc S,\\
q^4+3q^3+q^2+q+1 & S  \cap \Q(6,q) \cong P \cdot \Q^+(3,q) \text{ and } \pi_P \notinc S,\\
q^4+2q^3+q^2+q+1 & S \cap \Q(6,q) \cong m \cdot \Q(2,q) \text{ and } m \not\in \GH(q),\\
q^4 + q^3 + q^2 + q + 1 & S \cap \Q(6,q) \cong m \cdot \Q(2,q) \text{ and } m \in \GH(q).\\
\end{cases}
\]
\end{lem}

\begin{proof}
Let $x$ denote the number of points of $\Q(6, q)$ contained in $S$, and let $y$ denote the number of $\GH(q)$-lines contained in $S$. 
Then through each of the $x$ points of $S \cap \Q(6, q)$ there are $q + 1$ elements of $\mathcal{L}_1$ and  this way we have counted each of the $y$ lines exactly $q + 1$ times.
Therefore, $|\mathcal{L}_1| = x(q + 1) - yq$. 
We can now go through each case and plug in the values of $x$, $y$ to get the result. 
\begin{enumerate}
\item $S \cap \Q(6,q) \cong \Q(4,q)$: $x = q^3 + q^2 + 2 + 1$ and $y = q + 1$. 

\item $S \cap \Q(6,q) \cong P \cdot \Q^-(3,q)$: $x = q^3 + q + 1$ and $y = 1$.

\item $S \cap \Q(6,q) \cong P \cdot \Q^+(3,q) \text{ and } \pi_P \inc S$:  $x = q^3 + 2q^2 + q + 1$ and $y = (q + 1)^2$. 

\item $S  \cap \Q(6,q) \cong P \cdot \Q^+(3,q) \text{ and } \pi_P \notinc S$: $x = q^3 + 2q^2 + q + 1$ and $y = 2q + 1$. 

\item $S \cap \Q(6,q) \cong m \cdot \Q(2,q) \text{ and } m \not\in \GH(q)$: $x = q^3 + q^2 + q + 1$ and $y = q + 1$.

\item $S \cap \Q(6,q) \cong m \cdot \Q(2,q) \text{ and } m \in \GH(q)$: $x = q^3 + q^2 + q + 1$ and $y = q^2 + q + 1$. 
\end{enumerate}
\end{proof}

We now determine $|\mathcal{L}_1 \cap \mathcal{L}_2|$, which will finally allow us to compute the
sizes of the $1$-good structures obtained by our construction.  In the following lemmas we will
refer to the different cases by how they appear in Lemma \ref{lem:4dim_int}.  For example, ($c$)(i)
will refer to the case when $S \cap \Q(6,q) \cong P \cdot \Q^+(3,q)$ and $\pi_P \inc S$.  To reduce
the number of cases that we have to study, we first give a sufficient condition for $\mathcal{L}_1
\cap \mathcal{L}_2 = \mathcal{L}_2$.

\begin{lem}\label{L2inL1}
If $\pi_A$ is contained in $S$ or $S$ is contained in the tangent space $T_A$ at $A$, then
$\mathcal{L}_2 \subseteq \mathcal{L}_1$.
\end{lem}

\begin{proof}
Suppose $S$ lies in the tangent hyperplane $T_A$ at $A$, and suppose $n$ meets $\pi_A$ in at least a
point $X$. Then $AX$ is a line of $\GH(q)$ (since $\pi_A$ is an $\GH(q)$-plane) and hence
$\dist(n,A) \le 3$. This means that $n$ is contained in $T_A$ and so $n$ meets $S$ nontrivially as
$S$ is a hyperplane of $T_A$. Suppose $\pi_A$ is contained in $S$. Then any line $n$ which meets
$\pi_A$ nontrivially must meet $S$ nontrivially.
\end{proof}

It turns out that the converse of Lemma \ref{L2inL1} holds, which will be a consequence of the next
result.

\begin{lem}
\label{lem:size_L1L2}
Either $\mathcal{L}_2 \subseteq \mathcal{L}_1$ or
\[
|\mathcal{L}_1 \cap \mathcal{L}_2|=
\begin{cases}
2q^2 + q + 1 & (a) \text{ and }  \dim(\pi_A \cap S) = 1, \\
q^2 + 2q + 1 & (a) \text{ and }  \dim(\pi_A \cap S) = 0, \\
q^2+q+1 & (b) \text{ and } \dim(\pi_A \cap S) = 1, \\
2q+1 & (b) \text{ and } \dim(\pi_A \cap S) = 0, \\
2q^2 + 2q + 1 & (c) \text{ and } \dim(\pi_A \cap S) = 0,\\
2q^2 + q + 1 & (c),\, \dim(\pi_A \cap S) = 1 \text{ and } A \not\in \pi_P,\\
q^2 + q + 1 & (c),\, \dim(\pi_A \cap S) = 1 \text{ and } A \in \pi_P,\\
q^2 + 2q + 1 & (d) \text{ and } \dim(\pi_A \cap S) = 0,\\
q^2 + q + 1 & (d) \text{ and } \dim(\pi_A \cap S) =1.
\end{cases}
\]
\end{lem}

\begin{proof}
There are three types of lines in $\mathcal{L}_1 \cap \mathcal{L}_2$:
\begin{enumerate}[(1)]
\item $\GH(q)$-lines through $A$, 
\item $\GH(q)$-lines that intersect all those $\GH(q)$-lines through $A$ which are contained in $S$, and
\item for each point $X \in S \cap \Q(6,q)$ with $\dist(X,A) = 4$, the unique $\GH(q)$-line $\ell_X$
  through $X$ which is at distance $3$ from $A$.
\end{enumerate}
Therefore, for each case we just need to count the lines of type (2) which are not in (1) and the
lines of type (3) which are not in (2) (no line of type (3) can be of type (1)).  We will denote
these numbers by $\alpha$ and $\beta$ respectively, so that $|\mathcal{L}_1 \cap \mathcal{L}_2| = q
+ 1 + \alpha + \beta$.

By Lemma \ref{L2inL1}, we can assume that $\pi_A$ is not contained in $S$ and $S$ is not contained
in the tangent space $T_A$ at $A$.  Since $\pi_A$ is not contained in $S$, it must intersect $S$ in
the point $A$ or a line $\ell$ through $A$.  Therefore $\alpha = 0$ or $q^2$ depending on whether
$\pi_A \cap S$ is a point or a line.

To determine $\beta$ we will first compute the parameter $\gamma$ which is the total number of ideal
lines through $A$ contained in $S$.  This is because every point $X$ with $\dist(X,A) = 4$ is
collinear with $A$ in $\Q(6,q)$ with $AX$ being an ideal line.  Therefore, $q\gamma$ is the total
number of points at distance $4$ from $A$ which are contained in $S \cap \Q(6,q)$.  For two distinct
points $X, Y \in S \cap \Q(6,q)$ with $\dist(X,A) = \dist(Y,A) = 4$, we have $\ell_X = \ell_Y$ if
and only if $X$ and $Y$ are collinear in $\GH(q)$ and there is an $\GH(q)$-line $\ell$ through $A$
contained in $S$ which is concurrent with the $\GH(q)$-line $XY$ ($= \ell_X = \ell_Y$).  Therefore,
if $\pi_A \cap S = \{A\}$, then no line of type (3) is of type (2) and we have $\beta = q \gamma$.
In contrast, if $\pi_A \cap S$ is a line $\ell$ through $A$, then $\beta$ is equal to the total number
of points at distance $4$ from $A$ in $S \cap \Q(6,q)$ which are not incident with any of the
$\GH(q)$-line inside $S$ that is concurrent with $\ell$.  We now look at the different cases.

\begin{enumerate}[$(a)$]
\item $S \cap \Q(6,q) \cong \Q(4,q)$.\\ If $A$ lies on one of the $q + 1$ $\GH(q)$-lines in $S \cap
  \Q(6,q)$, then $\pi_A \cap S$ is that line through $A$, and otherwise $\pi_A \cap S = \{A\}$.
  Therefore, in the first case we have $\alpha = q^2$ and in the second $\alpha = 0$.  There are $q
  + 1$ totally singular lines in $\Q(4,q)$ through any point and no two $\GH(q)$-lines in $S$ are
  concurrent.  Therefore, in the first case, we have $\gamma = q$ and $\beta = q\gamma = q^2$, while
  in the second case we have $\gamma = q + 1$ and $\beta = q \gamma = q^2 + q$.  To summarise,
\begin{center}
\begin{tabular}{lllll}
\toprule
Case & $\alpha$ & $\gamma$ & $\beta$ & $|\mathcal{L}_1 \cap \mathcal{L}_2|$\\
\midrule
$\dim(\pi_A\cap S)=1$ & $q^2$ & $q$ & $q^2$ & $2q^2+q+1$\\
$\dim(\pi_A\cap S)=0$ & $0$ & $q+1$ & $q^2 + q$ & $q^2+2q+1$\\
\bottomrule
\end{tabular}
\end{center}

\item $S \cap \Q(6,q) \cong P \cdot \Q^-(3,q)$.\\ Let $\ell$ be the unique $\GH(q)$-line contained
  in $S$. The cases $\dim(\pi_A \cap S) = 1$ and $\dim(\pi_A \cap S) = 0$ correspond to
  $A\in\ell\setminus\{P\}$ and $A\notin\ell$, respectively.  In the former case, $\alpha=q^2$ and in
  the latter we have $\alpha = 0$.  In both cases, $A$ lies on a generator of the cone $P\cdot
  \Q^-(3,q)$ but is not the vertex of the cone.  So $A$ lies on a unique totally singular
  line contained in $S$. In the first case this means that $\gamma=0$ as $\ell\in\GH(q)$ and in the
  latter case it means $\gamma=1$.  Since there is only one $\GH(q)$-line in $S$, we must have
  $\beta = q\gamma$ in both cases.

\begin{center}
\begin{tabular}{lllll}
\toprule
Case & $\alpha$ & $\gamma$ & $\beta$ & $|\mathcal{L}_1 \cap \mathcal{L}_2|$\\
\midrule
$\dim(\pi_A \cap S) = 1$ & $q^2$ & $0$ & $0$ & $q^2+q+1$\\
$\dim(\pi_A \cap S) = 0$ & $0$ & $1$ & $q$ & $2q+1$\\
\bottomrule
\end{tabular}
\end{center}

\item $S \cap \Q(6,q) \cong P \cdot \Q^+(3,q)$.\\ We let $\pi_P$ be the unique $\GH(q)$-plane
  incident with $P$.  Note that the number of totally singular lines in $S$ incident with $A$ is
  $2q+1$ for every $A \neq P$ in $S \cap \Q(6,q)$.
\begin{description}
\item[$\pi_A$ meets $S$ in a point (namely $A$)] So $\alpha=0$ here.  We must be in case ($c$)(ii)
  of Lemma \ref{lem:4dim_int} and in fact $A$ is a point of $S \cap \Q(6,q)$ not lying on any of the
  two $\GH(q)$-planes contained in $S$.  Since none of the totally singular lines through $A$ in $S$
  are $\GH(q)$-lines, we have $\gamma = 2q + 1$.  Moreover, we have $\beta = q \gamma = 2q^2 + q$
  since $\pi_A \cap S$ is a point here.

\item[$\pi_A$ meets $S$ in a line and $A\notin\pi_P$] Let $\ell = \pi_A \cap S$.  Then $\ell$ is not
  incident with $P$ because $P\notin \pi_A$.  The parameter $\alpha$ is equal to $q^2$.  Since the
  line $\ell$ is the unique totally singular line through $A$ in $S$ which is contained in $\GH(q)$,
  and there are in total $2q + 1$ totally singular lines through $A$ in $S$, we have $\gamma = 2q$.
  This gives us $2q^2$ points of $S$ at distance $4$ from $A$.  The plane $\pi_1 = \langle P, \ell
  \rangle$ must be an $\GH(q)$-plane contained in $S$.  Denote the second plane through the line
  $PA$ contained in $S \cap \Q(6,q)$ by $\pi_2$.  Then for each of the $q^2$ points $X \in \pi_1
  \setminus \ell$, the unique $\GH(q)$-line $\ell_X$ with $\dist(A,\ell_X) = 3$ is already counted
  in the parameter $\alpha$.  The remaining $q^2$ points in $\pi_2 \setminus PA$ which are at
  distance $4$ from $A$ give rise to precisely $q^2$ elements of $\mathcal{L}_1 \cap \mathcal{L}_2$.
  Therefore, we have $\beta = q^2$.

\item[$\pi_A$ meets $S$ in a line and $A\in\pi_P$] Let $\ell = \pi_A \cap S$ and note that $P$ is
  incident with $\ell$.  We again have $\alpha = q^2$.  The $2q + 1$ totally singular lines through
  $A$ that are contained in $S$ all lie in the two totally singular planes $\pi_1$ and $\pi_2$
  through the line $\ell$.  Since these planes contain the $\GH(q)$-line $\ell$, they must be
  $\GH(q)$-planes with centres on the line $\ell$.  This shows that $\gamma = 0$ since for any $X
  \in (\pi_1 \cup \pi_2) \setminus \ell$ the unique $\GH(q)$-line $\ell_X$ at distance $3$ from $A$ is
  contained in $S$.  Therefore, $\beta = 0$.
\end{description}

\begin{center}
\begin{tabular}{lllll}
\toprule
Case & $\alpha$ & $\gamma$ & $\beta$ & $|\mathcal{L}_1 \cap \mathcal{L}_2|$\\
\midrule
$\dim(\pi_A\cap S)=0$ & $0$ & $2q + 1$ &$2q^2+q$ & $2q^2+2q+1$\\
$\dim(\pi_A\cap S)=1$ and $A\notin \pi_P$& $q^2$ & $2q$  & $q^2$ & $2q^2+q+1$\\
$\dim(\pi_A\cap S)=1$ and $A\in \pi_P$& $q^2$ & $0$ &  $0$ &$q^2+q+1$\\
\bottomrule
\end{tabular}
\end{center}

\item $S\cap \Q(6,q)=m\cdot \Q(2,q)$. \\ Say $\dim(\pi_A \cap S) = 0$.  Then $\alpha = 0$.
  Moreover, $\beta = q \gamma$, where $\gamma$ is the total number of ideal lines through $A$
  contained in $S$.  There are $q + 1$ totally singular lines through $A$ contained in $S$ (all
  lines joining $A$ to a point of $m$) and none of these are $\GH(q)$-lines.  Therefore, $\gamma =
  q+1$ and $\beta = q^2 + q$.

Say $\dim(\pi_A \cap S) = 1$ and let $\ell = \pi_A \cap S$.  Here we have $\alpha = q^2$.  Since all
the $q + 1$ totally singular lines through $A$ contained in $S$ lie in the $\GH(q)$-plane $\langle
\ell, m \rangle$, we must have $\gamma = 0$, which gives us $\beta = 0$.

\begin{center}
\begin{tabular}{lllll}
\toprule
Case & $\alpha$ & $\gamma$ & $\beta$ & $|\mathcal{L}_1 \cap \mathcal{L}_2|$\\
\midrule
$\dim(\pi_A\cap S)=0$  & $0$ & $q + 1$ & $q^2 +q$ & $q^2+2q+1$\\
$\dim(\pi_A\cap S)=1$ & $q^2$ & $0$ & $0$ & $q^2+q+1$\\
\bottomrule
\end{tabular}
\end{center}
\end{enumerate}
\end{proof}

\begin{cor}
The possible sizes of the $1$-good structures that can be obtained via this construction are $q^4 + q^3 + q^2 + q + 1 + k$ for 
$
k\in \{ 0, q^3-q, q^3, q^3+q^2-q, 2q^3-q^2-q, 2q^3-q^2, 2q^3-q, 2q^3, 3q^3-q^2-q, 3q^3-q^2, 3q^3\}.
$
\end{cor}
\begin{proof}
Starting with one of the cases of Lemma \ref{lem:size_L1} and then using the appropriate case of
Lemma \ref{lem:size_L1L2}, we can calculate the different possibles values of $|\mathcal{L}_1 \cup
\mathcal{L}_2| = |\mathcal{L}_1| + | \mathcal{L}_2| - |\mathcal{L}_1 \cap \mathcal{L}_2|$, noting
that $|\mathcal{L}_2| = q^3 + q^2 + q + 1$ in all cases.
\end{proof}

\section{Computer classifications}
\label{sec:computer}

In this section we describe an exhaustive search for $1$-good structures of the generalized quadrangles $\sympl(3,3)$, $\sympl(3,4)$ and $\sympl(3,5)$ and for the $1$-good structures in
$\sympl(3,7)$ that have a non-trivial automorphism group.

For many combinatorial problems in finite geometry there is a natural hierarchical property (that is, closed under taking subsets) for which the desired objects are just the extremal objects with that property. For example, in searching for hyperovals in a projective plane of order $q$, the property ``no $3$ points collinear'' is hierarchical. Therefore a backtrack search can proceed simply by adding points to the ``partial ovals'' until the desired size ($q+2$) is reached. This general situation has been heavily studied and there is a wide range of techniques for improving the efficiency of such a search, and exploiting symmetry to reduce the generation of isomorphic objects.

The situation for $1$-good structures is somewhat different in that there is no obvious property that can play a similar role --- it is not clear what the defining property of a ``partial $1$-good structure" should be. The naive choice of constructing sets of points and lines that induce a subgraph of maximum degree at most $q$ makes so little use of the constrained structure of a $q$-regular subgraph that it is essentially useless. 

However it is clear that when the search process branches by making a decision on a particular element (i.e., the first branch assumes the element is in, while the second assumes it is out), then there is a cascade of consequences affecting the other elements. This feature of the problem
led us to consider the \emph{constraint satisfaction programming} (CSP) paradigm. In a CSP problem, the user creates a \emph{declarative model} of a problem where the desired configurations are defined by a collection of \emph{constraints} that all solutions must satisfy. This is declarative (rather than \emph{imperative}) because the user just states the desired outcome but leaves the details of the search procedure to the CSP solver. The CSP solver is optimised precisely for the task of \emph{constraint propagation} thereby chasing down the effect of the entire cascade of consequences arising from the decision at each branch. Experience in related combinatorial searches has 
shown that in certain situations a general-purpose CSP solver can outperform even very heavily optimised bespoke programs.

In practice, for our computations we used the CSP solver {\sc Minion} \cite{minion} (mostly for familiarity, we have no reason to believe that it is either better or worse than other CSP solvers for this problem). The model contains two boolean arrays (one for points, one for lines) so that the variable corresponding to a point is true if the point is in the $1$-good structure and false otherwise. The constraints all have the form ``If $P$ is \emph{not in} the $1$-good structure then exactly one line through $P$ \emph{is in} the $1$-good structure'' (and the duals involving lines). This of course relies on the ability of one constraint to involve the ``status'' of another constraint. This ability is called \emph{reification} which, loosely speaking, allows a constraint to have the form ``If $A$ is satisfied, then $B$ must be satisfied'' where $A$ and $B$ are themselves constraints. A solution to the entire CSP may leave $A$ unsatisfied, in which case $B$ may or may not be satisfied without violating the {\tt reifyimply} constraint ({\tt reifyimply} is a {\sc Minion} construct).

One advantage of using CSP is that the model itself is almost trivial to write, while the complicated part (chasing consequences) is handed over to a program that has many users and has been maintained for many years, dramatically reducing the chances that it harbours a bug  sufficiently serious to compromise our results. 
Of course, as described, this CSP program will find numerous unnecessary isomorphs of \emph{every} $1$-good structure and so we also used a number of standard symmetry-breaking techniques such as specifying that particular small subsets of elements are in (or not in) the $1$-good structure. As these are entirely routine, we do not describe them further here.

When searching for $1$-good structures with specified automorphism group, say $H$, the model is very similar, except that there is one boolean variable for each \emph{orbit} of $H$. The constraints then must accommodate the possibility that each point in a point-orbit may be incident with more than one line in a line-orbit, and dually, but again this is a simple modification of the basic program.

We have provided the GAP code for this computation in the Appendix of this paper. 

\subsection{Tables}

The next four tables include summary information about the $q$-regular induced subgraphs (and
hence $1$-good structures) found in the generalized quadrangles $\sympl(3,q)$ for $q \in \{3,4,5,7\}$. For $q < 7$, the tables provide a complete listing of every $q$-regular subgraph / $1$-good structure in the
generalized quadrangle. For $q=7$ the computer searches assumed the existence of a non-trivial stabiliser group and so may miss those whose stabiliser group is trivial (although we do not believe there are any in $\sympl(3,7)$). Each line gives the size of a $q$-regular subgraph, the
order of its stabiliser (in the collineation group of $\sympl(3,q)$) and the orbits of this
stabiliser on both the $q$-regular subgraph and its complementary $1$-good structure.
The fifth column indicates when the example arises from Theorem  \ref{thm:construction}.
We note that some of the examples that do not arise from Theorem \ref{thm:construction}, are covered by the constructions given in \cite{BM11}. 
\renewcommand{\arraystretch}{1.15}

\begin{table}[!ht]
\begin{center}
\begin{tabular}{lrllc}
\toprule
\multicolumn{1}{c}{Size} & \multicolumn{1}{c}{Stabiliser size} & \multicolumn{1}{c}{Orbits (on subgraph)} & \multicolumn{1}{c}{Orbits (on $1$-good structure)}&From \ref{thm:construction}\\
\midrule
$36$ & $24$ & $\{2,4,6,12^{2}\}$&$\{1,2^{3},3,4,6,12^{2}\}$\\
$40$ & $12$ & $\{2^{2},6^{4},12\}$&$\{1^{5},2,3^{3},6^{4}\}$\\
$40$ & $240$ & $\{20^{2}\}$&$\{10^{2},20\}$\\
$42$ & $12$ & $\{1,2,3,6^{2},12^{2}\}$&$\{1^{2},3^{4},6^{4}\}$\\
$42$ & $36$ & $\{3,9,12,18\}$&$\{1^{3},2,3,6^{2},9^{2}\}$&\checkmark \\
$48$ & $36$ & $\{6^{2},18^{2}\}$&$\{1^{4},2^{2},3^{2},9^{2}\}$&\checkmark\\
$48$ & $36$ & $\{6^{2},18^{2}\}$&$\{1,2^{2},6^{3},9\}$\\
$48$ & $144$ & $\{24^{2}\}$&$\{1,3,4^{2},8,12\}$&\checkmark \\
$54$ & $36$ & $\{3,6,9,18^{2}\}$&$\{1^{3},2,3^{3},6^{2}\}$\\
$54$ & $36$ & $\{3,6,9,18^{2}\}$&$\{1^{3},2,3,9^{2}\}$\\
$54$ & $324$ & $\{27^{2}\}$&$\{1^{2},3^{2},9^{2}\}$&\checkmark\\
$56$ & $48$ & $\{4,12,16,24\}$&$\{2,4^{2},6,8\}$\\
$56$ & $48$ & $\{4,12,16,24\}$&$\{1,3,4,8^{2}\}$\\
\bottomrule
\end{tabular}
\end{center}
\caption{All induced $3$-regular subgraphs in $\sympl(3,3)$}
\end{table}

\begin{table}[!ht]
\begin{center}
\begin{tabular}{lrllc}
\toprule
\multicolumn{1}{c}{Size} & \multicolumn{1}{c}{Stabiliser size} & \multicolumn{1}{c}{Orbits (on subgraph)} & \multicolumn{1}{c}{Orbits (on $1$-good structure)}&From \ref{thm:construction}\\
\midrule
$100$ & $240$ & $\{10,20,30,40\}$&$\{5,10,15,20^{2}\}$\\
$100$ & $400$ & $\{50^{2}\}$&$\{10^{2},25^{2}\}$\\
$104$ & $96$ & $\{4,12,16,24,48\}$&$\{1^{3},3,4^{2},12,16,24\}$&\checkmark\\
$108$ & $144$ & $\{18^{2},36^{2}\}$&$\{4^{2},6^{2},9^{2},12^{2}\}$\\
$112$ & $192$ & $\{8,24,32,48\}$&$\{1,2,3,6^{2},16,24\}$&\checkmark\\
$112$ & $192$ & $\{8,24,32,48\}$&$\{1^{2},2^{2},4,8^{2},16^{2}\}$&\checkmark\\
$120$ & $96$ & $\{4,8,12,16,32,48\}$&$\{1^{3},3,4^{5},8,16\}$\\
$120$ & $120$ & $\{30^{2},60\}$&$\{2,3,5,10,15^{2}\}$\\
$120$ & $288$ & $\{12^{2},48^{2}\}$&$\{1^{4},3^{2},4^{2},16^{2}\}$&\checkmark\\
$120$ & $1440$ & $\{60^{2}\}$&$\{1,4,5^{2},15,20\}$&\checkmark\\
$128$ & $192$ & $\{16,48,64\}$&$\{1^{2},4^{2},16^{2}\}$\\
$128$ & $288$ & $\{4,12,16,48^{2}\}$&$\{1^{3},3,4^{3},12^{2}\}$\\
$128$ & $384$ & $\{16,48,64\}$&$\{1,2,3,8,12,16\}$\\
$128$ & $4608$ & $\{64^{2}\}$&$\{1^{2},4^{2},16^{2}\}$&\checkmark\\
$136$ & $136$ & $\{68^{2}\}$&$\{17^{2}\}$\\
\bottomrule
\end{tabular}
\end{center}
\caption{All induced $4$-regular subgraphs in $\sympl(3,4)$}
\end{table}

\begin{table}[!ht]
\begin{center}
\begin{tabular}{lrllc}
\toprule
\multicolumn{1}{c}{Size} & \multicolumn{1}{c}{Stabiliser size} & \multicolumn{1}{c}{Orbits (on subgraph)} & \multicolumn{1}{c}{Orbits (on $1$-good structure)}&From \ref{thm:construction}\\
\midrule
$230$ & $200$ & $\{5,10,25,40,50,100\}$&$\{1^{3},2^{2},5,10^{2},25^{2}\}$&\checkmark\\
$240$ & $100$ & $\{5^{4},20,25^{4},100\}$&$\{1^{3},4,5^{8},25\}$\\
$240$ & $200$ & $\{10^{2},20,50^{2},100\}$&$\{1,2,4,10^{4},25\}$\\
$240$ & $400$ & $\{20^{2},100^{2}\}$&$\{1,2,4,10^{2},20,25\}$\\
$240$ & $400$ & $\{20^{2},100^{2}\}$&$\{1^{4},4^{2},5^{2},25^{2}\}$&\checkmark\\
$240$ & $2400$ & $\{120^{2}\}$&$\{1,5,6^{2},24,30\}$&\checkmark\\
$250$ & $200$ & $\{5,10^{2},25,50^{2},100\}$&$\{1^{3},2^{2},5,25^{2}\}$\\
$250$ & $200$ & $\{5,10^{2},25,50^{2},100\}$&$\{1^{3},4,5^{3},10^{2},20\}$\\
$250$ & $400$ & $\{5,20,25,100^{2}\}$&$\{1^{3},4,5^{3},20^{2}\}$\\
$250$ & $10000$ & $\{125^{2}\}$&$\{1^{2},5^{2},25^{2}\}$&\checkmark\\
$252$ & $96$ & $\{6^{2},24^{2},96^{2}\}$&$\{1^{2},4,6,24^{2}\}$\\
\bottomrule
\end{tabular}
\end{center}
\caption{All induced $5$-regular subgraphs in $\sympl(3,5)$}
\end{table}

\begin{table}[!ht]
\begin{center}
\begin{tabular}{lrllc}
\toprule
\multicolumn{1}{c}{Size} & \multicolumn{1}{c}{Stabiliser size} & \multicolumn{1}{c}{Orbits (on subgraph)} & \multicolumn{1}{c}{Orbits (on $1$-good structure)}&From \ref{thm:construction}\\
\midrule
$658$ & $588$ & $\{7,14^{2},49,84,98^{2},294\}$&$\{1^{3},2^{3},7,14^{2},49^{2}\}$&\checkmark\\
$672$ & $294$ & $\{7^{6},42,49^{6},294\}$&$\{1^{3},6,7^{10},49\}$\\
$672$ & $294$ & $\{7^{6},42,49^{6},294\}$&$\{1^{3},6,7^{10},49\}$\\
$672$ & $294$ & $\{7^{6},42,49^{6},294\}$&$\{1^{3},6,7^{10},49\}$\\
$672$ & $294$ & $\{7^{6},42,49^{6},294\}$&$\{1^{3},6,7^{10},49\}$\\
$672$ & $588$ & $\{14^{3},42,98^{3},294\}$&$\{1,2,6,14^{5},49\}$\\
$672$ & $588$ & $\{14^{3},42,98^{3},294\}$&$\{1^{3},6,7^{4},14^{3},49\}$\\
$672$ & $1764$ & $\{42^{2},294^{2}\}$&$\{1^{4},6^{2},7^{2},49^{2}\}$&\checkmark\\
$672$ & $1764$ & $\{42^{2},294^{2}\}$&$\{1,2,6,14^{2},42,49\}$\\
$672$ & $14112$ & $\{336^{2}\}$&$\{1,7,8^{2},48,56\}$&\checkmark\\
$686$ & $588$ & $\{7,14^{3},49,98^{3},294\}$&$\{1^{3},6,7^{3},14^{3},42\}$\\
$686$ & $588$ & $\{7,14^{3},49,98^{3},294\}$&$\{1^{3},6,7^{3},14^{3},42\}$\\
$686$ & $588$ & $\{7,14^{3},49,98^{3},294\}$&$\{1^{3},2^{3},7,49^{2}\}$\\
$686$ & $882$ & $\{7,21^{2},49,147^{2},294\}$&$\{1^{3},6,7^{3},21^{2},42\}$\\
$686$ & $1764$ & $\{7,42,49,294^{2}\}$&$\{1^{3},6,7^{3},42^{2}\}$\\
$686$ & $86436$ & $\{343^{2}\}$&$\{1^{2},7^{2},49^{2}\}$&\checkmark\\
$688$ & $144$ & $\{8^{2},24^{2},48,144^{4}\}$&$\{1^{2},3^{2},8,48^{2}\}$\\
\bottomrule
\end{tabular}
\end{center}
\caption{Induced $7$-regular subgraphs in $\sympl(3,7)$ with nontrivial stabiliser}
\end{table}

\section{Concluding remarks}
We have improved the upper bounds on the cage number $c(k, 8)$ when $k$ is an even power of a prime, and on $c(k, 12)$ when $k$ is an arbitrary prime power. 
Of course, we also need $k - 1$ to not be a prime power, as otherwise $c(k, g)$ is equal to the Moore bound. 
For $g = 8$, the smallest value of $k$ for which we have an improvement in the upper bounds is $k = 16$, and for $g = 12$ the smallest value is $k = 11$. 
While the tables in Section \ref{sec:computer} show that the known general constructions in $\sympl(3,q)$ (including Theorem \ref{thm:construction}) are best possible for $q \in \{3, 4, 5, 7\}$, we believe that in general this is not true. 
In both $\sympl(3,q)$ and $\GH(q)$, we found some computer examples that we are unable to explain via general constructions, which suggests that there is still room for improvement in these bounds, at least for some special value of $q$. 
It will also be interesting to explore $1$-good structures in non-classical generalized quadrangles and find better constructions than those in the classical case. 

We have not given any new constructions of $t$-good structures for $t > 1$ in this paper as we believe that it is much more difficult to give geometrical constructions of these objects that lead to improvements in the known bounds. 
Our constructions of Section \ref{sec:const} and \ref{sec:const2} do not seem to generalise for $t > 1$. 
We conclude with the following:

\begin{quote}
\textbf{Open problem}: Find $1$-good structures of size $q^4 + \omega(q^3)$ in $\GH(q)$. 
\end{quote}

\subsection*{Acknowledgements}
The first author acknowledges the support of the Australian Research Council Future Fellowship
FT120100036, which supported the visit of the second author to the University of Western Australia.
We thank the referees for their comments and suggestions that have greatly improved the
exposition of this paper.


\end{document}